\documentclass[11pt]{amsart}
\usepackage[margin=1in]{geometry}

\usepackage{amssymb}
\usepackage{amsthm}
\usepackage{amsmath}
\usepackage{mathrsfs}
\usepackage{amsbsy}
\usepackage[all]{xy}
\usepackage{bm}
\usepackage{hyperref}
\usepackage{tikz}
\usepackage{array}
\usepackage{float}
\usepackage{enumerate}
\usepackage{xcolor}
\usepackage{hhline}
\allowdisplaybreaks
\usepackage[noadjust]{cite}

\usepackage{caption}
\usepackage{tabu}
\usepackage{diagbox}

\newenvironment{enumerate*}%
  {\begin{enumerate}[(I)]%
    \setlength{\itemsep}{10pt}%
    \setlength{\parskip}{0pt}}%
  {\end{enumerate}}

\newtheorem{theorem}{Theorem}[section]
\newtheorem{proposition}[theorem]{Proposition}
\newtheorem{corollary}[theorem]{Corollary}

\newtheorem{lemma}[theorem]{Lemma}

\theoremstyle{definition}

\newtheorem{remark}[theorem]{Remark}
\newtheorem{example}[theorem]{Example}

\DeclareMathOperator{\area}{area}
\DeclareMathOperator{\width}{width}
\DeclareMathOperator{\height}{height}
\DeclareMathOperator{\Regions}{Regions}
\DeclareMathOperator{\Loops}{Loops}
\DeclareMathOperator{\length}{length}
\DeclareMathOperator{\lat}{lat}
\DeclareMathOperator{\lon}{long}
\DeclareMathOperator{\vslice}{\mathsf{vslice}}
\DeclareMathOperator{\hslice}{\mathsf{hslice}}
\DeclareMathOperator{\ver}{vert}
\DeclareMathOperator{\hor}{hor}

\newcommand{\dfn}[1]{\textcolor{blue}{\emph{#1}}}

\begin{document}

\title[]{Loops and Regions in Hitomezashi Patterns}
\subjclass[2010]{}

\author[Colin Defant and Noah Kravitz]{Colin Defant}
\address[]{Department of Mathematics, Princeton University, Princeton, NJ 08540, USA}
\email{cdefant@princeton.edu}
\author[]{Noah Kravitz}
\address[]{Department of Mathematics, Princeton University, Princeton, NJ 08540, USA}
\email{nkravitz@princeton.edu}

\maketitle

\begin{abstract}
\emph{Hitomezashi patterns}, which originate from traditional Japanese embroidery, are intricate arrangements of unit-length line segments called \emph{stitches}. The stitches connect to form \emph{hitomezashi strands} and \emph{hitomezashi loops}, which divide the plane into regions. We investigate the deeper mathematical properties of these patterns, which also feature prominently in the study of corner percolation.  It was previously known that every loop in a hitomezashi pattern has odd width and odd height.  We additionally prove that such a loop has length congruent to $4$ modulo $8$ and area congruent to $1$ modulo $4$.  Although these results are simple to state, their proofs require us to understand the delicate topological and combinatorial properties of \emph{slicing} operations that can be applied to hitomezashi patterns.  We also show that the expected number of regions in a random $m\times n$ hitomezashi pattern (chosen according to a natural random model) is asymptotically $\left(\frac{\pi^2-9}{12}+o(1)\right)mn$.

\end{abstract}

\section{Introduction}\label{SecIntro}

\emph{Sashiko} is a Japanese style of embroidery that developed during the Edo period (1603--1867) for both practical and decorative purposes. Its name, which translates to ``little stabs,'' refers to the act of stabbing a needle through cloth. One type of sashiko artwork known as \emph{hitomezashi} (translating to ``one-stitch'') uses simple rules to create beautiful geometric patterns in a rectangular grid (see Figure~\ref{FigHito1}). We became aware of hitomezashi stitching patterns through a video \cite{Numberphile} on Brady Haran's YouTube channel \emph{Numberphile}. The video stars Ayliean MacDonald, who describes how to create hitomezashi patterns with a pen and paper and also suggests several interesting questions.  We encourage you to go watch the video; we'll be here when you come back. 

\begin{figure}[ht]
  \begin{center}\includegraphics[height=4.59cm]{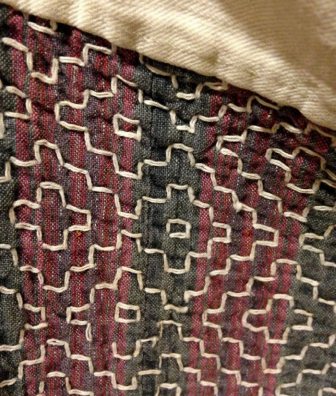}\qquad\qquad\qquad \includegraphics[height=4.59cm]{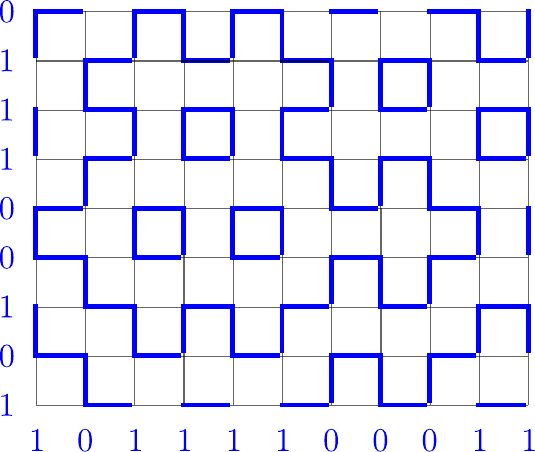}
  \end{center}
  \caption{On the left is a photo taken by Carol Hayes of a hitomezashi pattern on an apron from the Amuse Museum collection \cite{HayesSeaton}. On the right is an $8\times 10$ hitomezashi pattern. }\label{FigHito1}
\end{figure}

In case you didn't watch the video, we'll describe the procedure here. Begin with an $m\times n$ rectangular grid. Let us imagine that the grid is embedded in $\mathbb R^2$ so that the vertices are the points $(i,j)\in\mathbb Z^2$ with $0\leq i\leq n$ and $0\leq j\leq m$. We use the words \emph{north}, \emph{south}, \emph{east}, and \emph{west} to indicate directions in the grid, and we refer to the $x$-coordinate of a vertical line or line segment (respectively, the $y$-coordinate of a horizontal line or line segment) as its \dfn{longitude} (respectively, \dfn{latitude}).
We also coordinatize unit-length line segments via their midpoints. Thus, for integers $i$ and $j$, the vertical line segment between the points $(i,j)$ and $(i,j+1)$ is said to have longitude $i$ and latitude $j+1/2$. Similarly, the horizontal line segment between the points $(i,j)$ and $(i+1,j)$ has longitude $i+1/2$ and latitude $j$. We write $\lon(s)$ and $\lat(s)$ for the longitude and latitude, respectively, of a unit-length segment $s$.

Label each of the $m+1$ horizontal grid lines and each of the $n+1$ vertical grid lines with either a $0$ or a $1$. Suppose the vertical grid line with longitude $i$ has label $\epsilon_i\in\{0,1\}$. Draw all of the line segments with longitude $i$ that lie inside the grid and have latitude $j+1/2$ for some $j\equiv \epsilon_i\pmod 2$. We call these line segments \dfn{vertical stitches} because they are meant to represent the parts of the thread that pass over the cloth. We do not draw the line segments of longitude $i$ and latitude $j+1/2$ for $j\equiv \epsilon_i+1\pmod 2$ because these represent the parts of the thread that pass underneath the cloth. Repeat this process for every $0\leq i\leq n$. Similarly, if the horizontal grid line with latitude $j$ has label $\eta_j$, then we draw all of the line segments with latitude $j$ and longitude $i+1/2$ for some $i\equiv\eta_j\pmod 2$. These line segments are called \dfn{horizontal stitches}. We illustrate this construction on the right in Figure~\ref{FigHito1}.  The end product is called an \dfn{$m \times n$ hitomezashi pattern}.  It is often mathematically convenient to consider the infinite configuration that results from applying this procedure to the entire plane instead of an $m\times n$ grid. Such an infinite configuration is called simply a \dfn{hitomezashi pattern}.

\begin{figure}[ht]
 \centering
 \includegraphics[height=8.5cm]{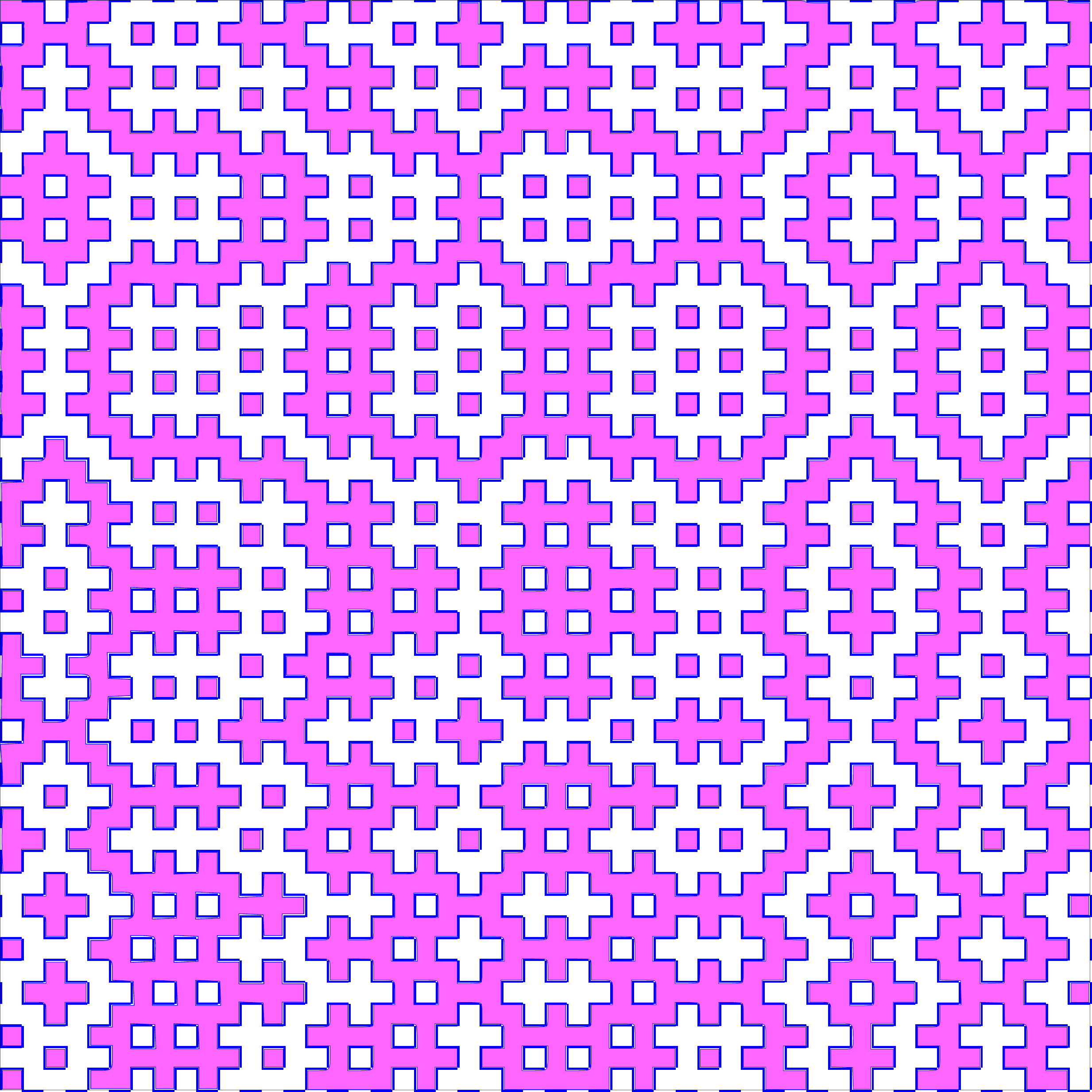}
  \caption{A $50\times 50$ hitomezashi pattern generated by choosing the grid line labels randomly. We have colored the regions in pink and white.}\label{FigHito4}
\end{figure}

Hitomezashi stitching has become a popular example of mathematics-inspired art.  These patterns are also appealing to mathematicians because they seem to have deeper mathematical properties---even a quick glance at the examples above suggests many types of structure and symmetry.  (See also \cite{Holden} for a mathematical analysis of other types of embroidery.)  Our goal in this article is to identify some of these properties and analyze them in precise terms.

Coming from a completely different direction, Pete~\cite{Pete} studied (something equivalent to) Hitomezashi patterns in the context of percolation  theory; we will say more about his work shortly.

A \dfn{hitomezashi loop} is a closed curve formed from stitches in a hitomezashi pattern. Hitomezashi loops are the boundaries of special types of \emph{polyominoes}, which are combinatorial geometric objects that have appeared prominently in recreational mathematics, combinatorics, and statistical physics \cite{Barequet, Conway, ConwayGuttmann, Gardner1960, Golomb, GuttmannBook, Jensen, Klarner}. Despite the simplicity of the procedure used to create hitomezashi patterns, a thorough understanding of all possible hitomezashi loops appears quite nontrivial. In \cite{HayesSeaton}, Hayes and Seaton observe that each of the cross-sections of the hitomezashi loop 
\[\begin{array}{l}\includegraphics[height=1.55cm]{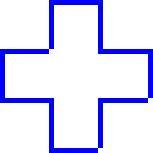}\end{array}\] contains an odd number of boxes; they remark that this is significant because ``even numbers are considered less favorably than odd numbers in Japanese culture.'' In Section~\ref{sec:basic}, we will use a simple parity argument to show that this property of cross-sections holds for all hitomezashi loops (see Proposition~\ref{prop:cross-section}).

In the original version of this article, our main effort was proving that every hitomezashi loop has odd width and odd height. In light of the aforementioned comment about Japanese culture, this shows that hitomezashi loops should be viewed favorably! However, we subsequently learned that some of our results were already discovered in an article by Pete \cite{Pete}, who used the name \emph{corner percolation} instead of \emph{hitomezashi patterns}. It appears that until now, people interested in corner percolation were unaware of hitomezashi and vice versa.  By interpreting hitomezashi loops in terms of pairs of Dyck paths, Pete established several structural results, and an immediate consequence of his work is that all hitomezashi loops have odd width and odd height.

\begin{theorem}[Pete~\cite{Pete}]\label{ThmOdd}
Every hitomezashi loop has odd width and odd height.
\end{theorem}

Rather than reprove Pete's results, we will simply describe his work in Section~\ref{sec:pete}.


A simple closed curve $C$ partitions the plane into a bounded region called the \dfn{interior} of $C$ and an unbounded region called the \dfn{exterior} of $C$. We write $\overline C$ for the union of $C$ with its interior. We write $\area(A)$ for the area of a region $A$ in the plane. Given a hitomezashi loop $L$, we let $\length(L)$ denote its \dfn{length} (i.e., number of stitches) and let $\area(L)$ denote the area of its interior; by slight abuse of terminology, we will also refer to $\area(L)$ as the \dfn{area} of $L$ itself.  It is natural to ask about the possible lengths and areas of hitomezashi loops.  Our next two theorems completely answer these questions.

\begin{theorem}\label{thm:possible-length}
Every hitomezashi loop has length congruent to $4$ modulo $8$.
\end{theorem}

\begin{theorem}\label{thm:possible-area}
Every hitomezashi loop has area congruent to $1$ modulo $4$.
\end{theorem}

\begin{remark}
Suppose $m,n\geq 3$ are odd integers, and consider the $m\times n$ hitomezashi pattern given by $\epsilon_0=\epsilon_n=\eta_0=\eta_m=1$ and all other grid labels equal to $0$. Let $R_{m,n}$ denote the rug-shaped loop in this pattern that touches the boundary of the grid (see Figure~\ref{FigHito16}). We also define $R_{1,1}$ to be the square-shaped hitomezashi loop with $4$ stitches. Then $\width(R_{m,n})=n$, $\height(R_{m,n})=m$, $\length(R_{m,n})=4m+4n-12$ (except for $\length(R_{1,1})=4$), and $\area(R_{m,n})=mn-m-n+2$.  These examples show that all of the widths, heights, lengths, and areas not ruled out by Theorems~\ref{ThmOdd}--\ref{thm:possible-area} actually occur.
\end{remark}

\begin{figure}[ht]
 \centering
 \includegraphics[height=2.94cm]{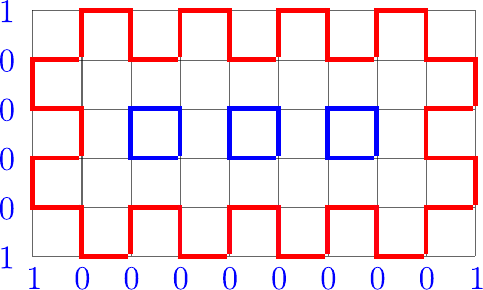}
  \caption{The rug-shaped hitomezashi loop $R_{5,9}$ (in red).}\label{FigHito16}
\end{figure}

We were surprised to discover that although Theorem~\ref{ThmOdd} has a short proof, the structural properties described in Theorems~\ref{thm:possible-length} and~\ref{thm:possible-area} are unexpectedly nontrivial.  Our proof method uses delicate inductive arguments that employ both topological and combinatorial ideas.  In Section~\ref{sec:slice}, we will explore a natural operation on hitomezashi patterns called \emph{slicing}.  Then, in Section~\ref{sec:length-area}, we will use slicing to prove Theorems~\ref{thm:possible-length} and~\ref{thm:possible-area}.

\begin{figure}[ht]
 \centering
 \includegraphics[height=4.55cm]{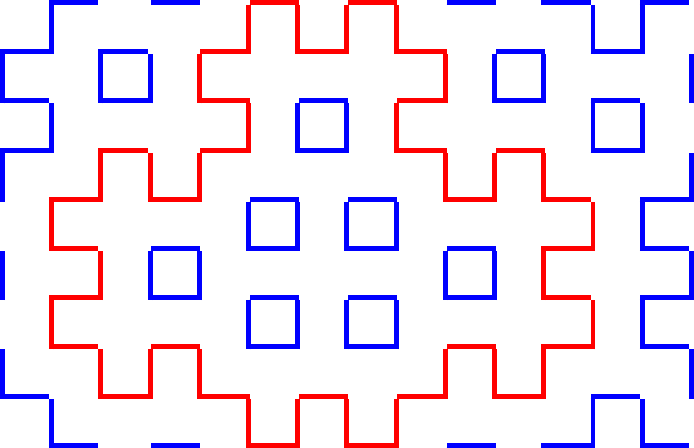}
  \caption{A $9 \times 14$ hitomezashi pattern with several loops, one of which is colored red.}\label{FigHito3}
\end{figure}

\begin{example}
We illustrate the preceding three theorems with the red hitomezashi loop $L$ shown in Figure~\ref{FigHito3}. The width and height of $L$ are $11$ and $9$, which are both odd. We have $\length(L)=60\equiv 4\pmod 8$ and $\area(L)=57\equiv 1\pmod 4$.
\end{example}

The stitches within an $m\times n$ hitomezashi pattern divide the $m\times n$ grid into regions; these regions are colored in pink and white in Figure~\ref{FigHito4}. Let $\Regions(H)$ denote the number of regions of an $m \times n$ hitomezashi pattern $H$. In the \emph{Numberphile} video \cite{Numberphile}, MacDonald suggests studying what happens when the $0/1$-labels of the horizontal and vertical grid lines are chosen randomly. Let ${\mathcal H}_{p}(m,n)$ denote the random $m\times n$ hitomezashi pattern generated by choosing all the labels independently at random so that each label is $0$ with probability $p$. For the sake of simplicity, we focus on the case where $p=1/2$. It is natural to ask for $\mathbb E[\Regions({\mathcal H}_{1/2}(m,n))]$, the expected number of regions that appear in this random model. Our next theorem addresses this. In order to state it, we define ${\bf L}$ to be the set of all hitomezashi loops modulo translation. 

\begin{theorem}\label{thm:expected-regions}
As $\min\{m,n\}\to\infty$, the expected number of regions in $\mathcal H_{1/2}(m,n)$ satisfies \[\mathbb E[\Regions({\mathcal H}_{1/2}(m,n))]=\left(\frac{\pi^2-9}{12}+o(1)\right)mn.\] 
\end{theorem}
We will prove Theorem~\ref{thm:expected-regions} in Section~\ref{sec:counting}.

Finally, in Section~\ref{sec:open}, we will gather several open questions and promising topics for future inquiry.

\section{Basic properties and additional terminology}\label{sec:basic}
Let us start by collecting a few basic facts about structures that can appear in hitomezashi patterns.  First, we wish to justify our discussion of ``curves'' and ``loops.''  Recall that a \emph{hitomezashi pattern} is by convention infinite.

\begin{proposition}\label{prop:curves}
Let $H$ be a hitomezashi pattern.  If $i$ and $j$ are integers, then there are exactly two stitches of $H$---one vertical and one horizontal---with $(i,j)$ as an endpoint.
\end{proposition}

To see this proposition, note that of the two potential horizontal stitches with endpoint $(i,j)$, exactly one is present; the same goes for the vertical stitches.  (In the same way, for an $m \times n$ hitomezashi pattern, every non-boundary lattice point has exactly two stitches passing through it, and every boundary lattice point has at most two stitches passing through it.)  This proposition tells us that in a hitomezashi pattern, we can trace our way unambiguously from one stitch to the next without ever reaching a ``branch point'' where we have to make a decision, and we alternately pass through horizontal and vertical stitches.  We define a \dfn{hitomezashi path} to be a curve (not necessarily maximal) formed from stitches in a hitomezashi pattern.  We obtain a loop if we eventually return to our starting point.  (So a hitomezashi loop is just a closed hitomezashi path.)  If we never return to our starting point, then our path continues indefinitely in both directions.  We use the term \dfn{strand} to refer to a maximal path (that is, either a loop or a bi-infinite path) in a hitomezashi pattern.  Note that every stitch in a hitomezashi pattern is contained in a unique strand. 

\begin{remark}
The strands of a hitomezashi pattern partition the plane into a (possibly infinite) number of regions that border each other along the various strands.  These regions can be $2$-colored so that neighboring regions receive different colors, and this $2$-coloring is unique up to exchanging the colors. Figure~\ref{FigHito4} shows such a $2$-coloring for a $50\times 50$ hitomezashi pattern.
\end{remark}

Our next observation is that the alternation of stitches at each latitude and longitude imposes parity restrictions on hitomezashi patterns.  We can give an \dfn{orientation} to a hitomezashi path by choosing a direction in which to traverse it. We often choose to orient a hitomezashi loop counterclockwise so that the interior is always on our left as we traverse it. 

\begin{lemma}[Parity argument]\label{Lem:Parity}
Let $P$ be an oriented hitomezashi path.  If $P$ has at least one horizontal stitch at latitude $y$,  then all of the stitches of $P$ at latitude $y$ are oriented in the same direction (west-to-east or east-to-west).  The analogous statement holds for vertical stitches (at a fixed longitude).
\end{lemma}

\begin{proof}
We give the argument for horizontal stitches since the vertical stitch case is identical.  Let $(x,y)$ be the coordinates of the initial end of some horizontal stitch $s$ at latitude $y$.  In order to get to the initial end $(x', y')$ of another horizontal stitch $s'$ in $P$, we have to traverse an even number of edges, since we alternate between passing through horizontal and vertical stitches.  This implies that $x+y$ and $x'+y'$ have the same parity.  If $s'$ also has latitude $y$, then $y=y'$, so $x$ and $x'$ have the same parity.  This means that the horizontal stitches through $(x,y)$ and $(x',y')$ either both extend to the west or both extend to the east, as we wanted to show.
\end{proof}

As promised in the introduction, we can use this lemma to show that hitomezashi loops have odd cross-sections.  We say that two stitches of $L$ with longitude $a$ are \dfn{consecutive} if there is no other stitch of $L$ with longitude $a$ and latitude between the latitudes of our chosen stitches.

\begin{proposition}\label{prop:cross-section}
Let $L$ be a hitomezashi loop.  Let $a$ be a longitude at which $L$ has at least one horizontal stitch.  Then the difference in latitude between any two consecutive stitches of $L$ at longitude $a$ is an odd integer.  Moreover, the difference in latitude between the north-most and south-most stitches of $L$ at longitude $a$ is an odd integer.  The analogous statements hold for vertical stitches.
\end{proposition}

\begin{proof}
Orient $L$ counterclockwise so that the interior of $L$ always lies to the left of $L$.  We prove only the statement for horizontal stitches.  First, let $s$ and $s'$ be consecutive (horizontal) stitches of $L$ at longitude $a$.  Then $s$ and $s'$ are oriented in different directions because the cells between $s$ and $s'$ are either all in the interior of $L$ or all in the exterior of $L$.  Let $t$ denote the (vertical) stitch of $L$ immediately before $s$, and let $t'$ denote the (vertical) stitch of $L$ immediately after $s'$; note that $\lon(t)=\lon(t')$.  So $\lat(t)-\lat(t')$ is an even integer.  By Lemma~\ref{Lem:Parity}, $t$ and $t'$ are oriented in the same direction, and we conclude that $\lat(s)-\lat(s')$ is an odd integer.

For the second statement of the proposition, let $s_1, \ldots, s_r$ denote the horizontal stitches of $L$ at longitude $a$, ordered from north to south.  Since $L$ is a closed loop, we see that $r$ must be even.  Each $\lon(s_i)-\lon(s_{i+1})$ is an odd integer (by the previous paragraph), so $$\lat(s_1)-\lat(s_r)=(\lat(s_1)-\lat(s_2))+(\lat(s_2)-\lat(s_3))+\cdots+(\lat(s_{r-1})-\lat(s_r))$$ is also an odd integer.
\end{proof}

The next proposition explains many of the ``mirroring'' symmetries that arise in hitomezashi patterns. If a stitch $s$ appears before a stitch $t$ as we traverse an oriented hitomezashi path $P$, then we write $P_{[s,t]}$ for the subpath of $P$ that starts at $s$ and ends at $t$ (including $s$ and $t$). We say an oriented hitomezashi path is \dfn{vertically monotone} (respectively, \dfn{horizontally monotone}) if all its vertical (respectively, horizontal) stitches are oriented in the same direction. 

\begin{lemma}[Mirroring]\label{lem:mirror}
Let $P$ be an oriented hitomezashi path. Let $s,t,t',s'$ be horizontal stitches that appear in this order (up to a cyclic shift if $P$ is a loop) as we traverse $P$, and assume $\lat(s)=\lat(s')$ and $\lat(t)=\lat(t')$. Suppose all vertical stitches in $P_{[s,t]}$ (respectively, $P_{[t',s']}$) are oriented from south to north (respectively, north to south). Then $P_{[t',s']}$ can be obtained by reflecting $P_{[s,t]}$ across a vertical axis (ignoring orientations).  The analogous statement for reflections across a horizontal axis also holds.
\end{lemma}

\begin{proof}
We illustrate the proof in Figure~\ref{FigHito6}.  Because $P_{[s,t]}$ and $P_{[t',s']}$ are both vertically monotone, these paths have the same number of stitches.  Let $s=u_1,u_2,\ldots,u_r=t$ be the stitches in $P_{[s,t]}$ given in the order of the orientation of $P$, and let $t'=v_r,v_{r-1},\ldots,v_1=s'$ be the stitches in $P_{[t',s']}$ given in the order of the orientation of $P$. Note that for each $1\leq i\leq r$, the stitches $u_i$ and $v_i$ have the same latitude. If $i$ is odd, then $u_i$ and $v_i$ are both horizontal, so, by Lemma~\ref{Lem:Parity}, they are oriented in the same direction (west-to-east or east-to-west). Let $a=(\lon(s)+\lon(s'))/2$. We claim that for each $1\leq i\leq r$, the stitch $v_i$ is obtained by reflecting $u_i$ across the line $x=a$ (ignoring orientations). We proceed by induction on $i$, noting that the base case $i=1$ follows immediately from the choice of $a$.  For the induction step, assume that $v_{i-1}$ is obtained by reflecting $u_{i-1}$ across $x=a$.
If $i$ is even, then $u_{i-1}$ and $v_{i-1}$ have the same orientation, and $u_i$ and $v_i$ are oriented south-to-north and north-to-south, respectively. It follows that $v_i$ is obtained by reflecting $u_i$ across $x=a$ in this case. If $i$ is odd, then $u_{i}$ and $v_{i}$ have the same orientation, and $u_{i-1}$ and $v_{i-1}$ are oriented south-to-north and north-to-south, respectively. The desired claim follows in this case as well.  The proof for horizontal reflections is identical.
\end{proof}

\begin{figure}[ht]
 \centering
 \includegraphics[height=5.233cm]{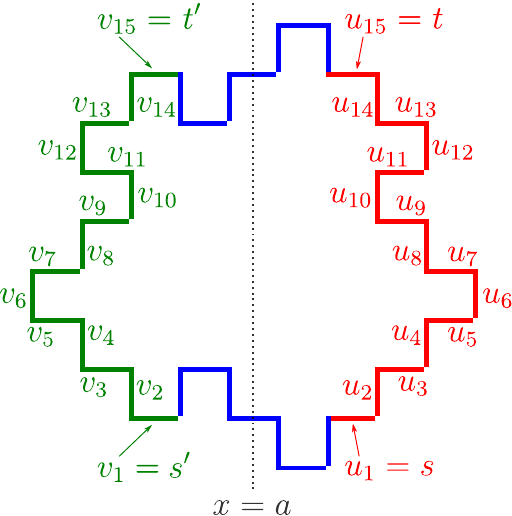}
  \caption{An illustration of the proof of Proposition~\ref{lem:mirror}.}\label{FigHito6}
\end{figure}


A special case of this mirroring phenomenon will be useful later.  Let $s$ be a horizontal stitch in a hitomezashi loop $L$, and let $s_1$ and $s_2$ denote the vertical stitches of $L$ on either side of $s$.  If $s_1$ and $s_2$ have the same latitude and the cell between $s_1$ and $s_2$ lies in the exterior of $L$, then we say that $s$ is an \dfn{indent} of $L$.  If $s_1$ and $s_2$ have the same latitude and the cell between $s_1$ and $s_2$ lies in the interior of $L$, then we say that $s$ is an \dfn{outdent} of $L$.  If $s_1$ and $s_2$ have different latitudes, then we say that $s$ is a \dfn{nondent} of $L$.  We also make the analogous definitions for vertical stitches.  The following corollary is a special case of Proposition~\ref{lem:mirror}.

\begin{corollary}\label{prop:outdents}
Let $L$ be a hitomezashi loop.  If $a$ is a longitude at which $L$ has at least one horizontal stitch, then the stitches of $L$ at longitude $a$ are either all indents, all outdents, or all nondents.  The analogous statement holds for vertical stitches.
\end{corollary}

In light of this corollary, we may classify the longitudes at which $L$ has horizontal stitches as \dfn{indent longitudes}, \dfn{outdent longitudes}, and \dfn{nondent longitudes}; we also do the same for latitudes.

\section{Pete's interpretation}\label{sec:pete}
We will now describe Pete's interpretation of Hitomezashi patterns in terms of Dyck paths, and we will outline the structural results that he deduces from this perspective.

Let $L$ be a hitomezashi loop, and consider its clockwise orientation.  Lemma~\ref{Lem:Parity} tells us that for each longitude at which $L$ has vertical stitches, all of the stitches of $L$ at this longitude are oriented in the same direction; say that such a longitude is an \dfn{up-longitude} if the stitches are oriented south-to-north, and that it is a \dfn{down-longitude} if the stitches are oriented north-to-south.  We obtain an up-down path $P'_{\ver}(L)$ consisting of $\width(L)+1$ steps by reading off the sequence of up- and down-longitudes from left to right and recording an up step for each down-longitude and a down step for each up-longitude. We then get a path $P_{\ver}(L)$ consisting of $\width(L)-1$ steps by deleting the first and last steps from $P_{\ver}'(L)$.  Likewise, identify each latitude intersecting $L$ as a \dfn{left-latitude} or a \dfn{right-latitude}. We obtain an up-down path $P'_{\hor}(L)$ consisting of $\height(L)+1$ steps by reading off the sequence of left- and right-latitudes from bottom to top and recording an up step for each left-latitude and a down step for each right-latitude. Deleting the first and last steps from $P'_{\hor}(L)$ yields a path $P_{\hor}(L)$.  See Figure~\ref{FigHito18}.  It is clear that the pair $(P_{\ver}(L), P_{\hor}(L))$ uniquely determines $L$; Pete's insight is that the possible arising pairs $(P_{\ver}(L), P_{\hor}(L))$ have a special structure.

\begin{figure}[ht]
 \centering
 \includegraphics[height=7.35cm]{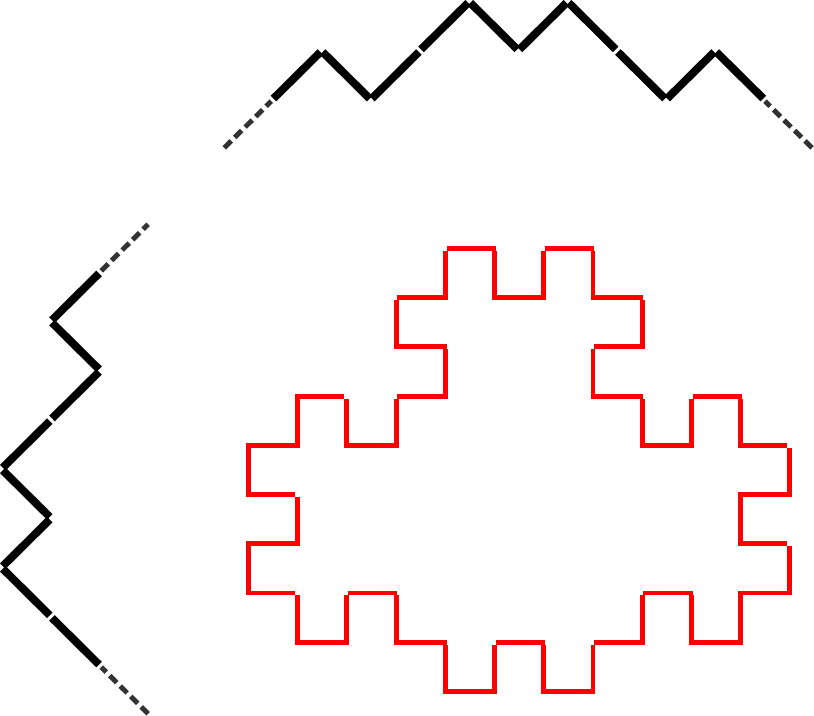}
  \caption{A hitomezashi loop $L$ and its corresponding Dyck paths $P_{\ver}(L)$ (top) and $P_{\hor}(L)$ (left). If we represent up steps by $\text{U}$ and down steps by $\text{D}$, then these paths are $P_{\ver}(L)=\text{UDUUDUDDUD}$ and $P_{\hor}(L)=\text{UUDUDDUD}$. Each of these paths has height $2$.}\label{FigHito18}
\end{figure}

Recall that a \dfn{Dyck path of semilength $n$} is a sequence of $n$ up-steps and $n$-down steps such that, reading from left to right, there are always at least as many up-steps as down-steps.  The \dfn{height} of a Dyck path is the maximum, over all prefixes of the sequence, of the difference between the number of up-steps and the number of down-steps. Recall that we write ${\bf L}$ for the set of hitomezashi loops modulo translation. 

\begin{theorem}[Pete~\cite{Pete}]\label{thm:pete}
The map $L\mapsto (P_{\ver}(L), P_{\hor}(L))$ is a bijection from ${\bf L}$ to the set of pairs of Dyck paths of the same height. Moreover, if $L$ has width $w$ and height $h$, then $P_{\ver}(L)$ and $P_{\hor}(L)$ have semilengths $(w-1)/2$ and $(h-1)/2$, respectively.
\end{theorem}

Pete proves this theorem by analyzing a suitable \emph{height function}; see~\cite{Pete} for more details.

Since the semilength of a Dyck path is always an integer, it is an immediate consequence (Theorem~\ref{ThmOdd}) that every hitomezashi loop has odd width and odd height.  Pete also extracts more detailed structural information from Theorem~\ref{thm:pete}.  Let us say a horizontal stitch in a hitomezashi loop $L$ is \dfn{north-extremal} (respectively, \dfn{south-extremal}) if its latitude is maxmimal (respectively, minimal) among all horizontal stitches in $L$; define \dfn{east-extremal} and \dfn{west-extremal} vertical stitches analogously. 

\begin{theorem}[Pete~\cite{Pete}]\label{ThmExtremalMatch}
Let $L$ be a hitomezashi loop. Then $L$ has a north-extremal horizontal stitch with longitude $a$ if and only if it has a south-extremal horizontal stitch with longitude $a$. Similarly, $L$ has an east-extremal vertical stitch with latitude $b$ if and only if it has a west-extremal vertical stitch with latitude $b$.
\end{theorem}

\begin{theorem}[Pete~\cite{Pete}]\label{ThmTwoStitches}
Let $L$ be a hitomezashi loop. If $L$ has a north-extremal or south-extremal horizontal stitch with longitude $a$, then it has exactly two stitches with longitude $a$. If $L$ has an east-extremal or west-extremal vertical stitch with latitude $b$, then it has exactly two stitches with latitude $b$. 
\end{theorem}

The set of Dyck paths of prescribed semilength and height is well-studied, and this will be useful in Section~\ref{sec:counting}.

\section{Slicing, splitting, and splicing}\label{sec:slice}
The goal of this section is to define \emph{slicing} operations that we can apply to hitomezashi patterns. These operations are useful because they often reduce the widths and/or heights of hitomezashi loops; this will allow us to prove structural results about hitomezashi loops via induction in the next two sections. 

Throughout this section, we fix a half-integer $i_0$ and assume that we have a hitomezashi pattern $H$ with labels $\epsilon_i$ and $\eta_j$ such that the two adjacent vertical grid lines with longitudes $i_0-1/2$ and $i_0+1/2$ have the same label: $\epsilon_{i_0-1/2}=\epsilon_{i_0+1/2}$. We can obtain a new hitomezashi pattern with two vertical grid lines removed as follows: We delete the vertical strip of the hitomezashi pattern between the longitudes $i_0-1$ and $i_0+1$ and then fill in the gap by sliding together the left and right parts of the remaining pattern.  We call this operation \dfn{vertical slicing} (at longitude $i_0$).  The new hitomezashi pattern $\widetilde{H}$ obtained from $H$ by vertical slicing at $i_0$ is defined by the following labels:
\begin{itemize}
    \item If $i<i_0$, then $\widetilde{\epsilon}_i=\epsilon_{i-1}$.
    \item If $i> i_0$, then $\widetilde{\epsilon}_i=\epsilon_{i+1}$.
    \item For all $j$, $\widetilde{\eta}_j=\eta_j$.
\end{itemize}
We will write $\vslice_{i_0}$ for the vertical slicing operator at $i_0$.  In the notation above, $\vslice_{i_0}(H)=\widetilde{H}$.  In the same way, $\vslice_{i_0}$ sends a set $\mathcal S$ of stitches of $H$ to its image $\vslice_{i_0}(\mathcal S)$ in $\vslice_{i_0}(H)$.  Of course, \dfn{horizontal slicing} works the same way, where we delete a horizontal strip instead of a vertical strip; we denote the horizontal slicing operator at latitude $j_0$ by $\hslice_{j_0}$. 

\begin{figure}[ht]
 \centering
 \includegraphics[height=8.25cm]{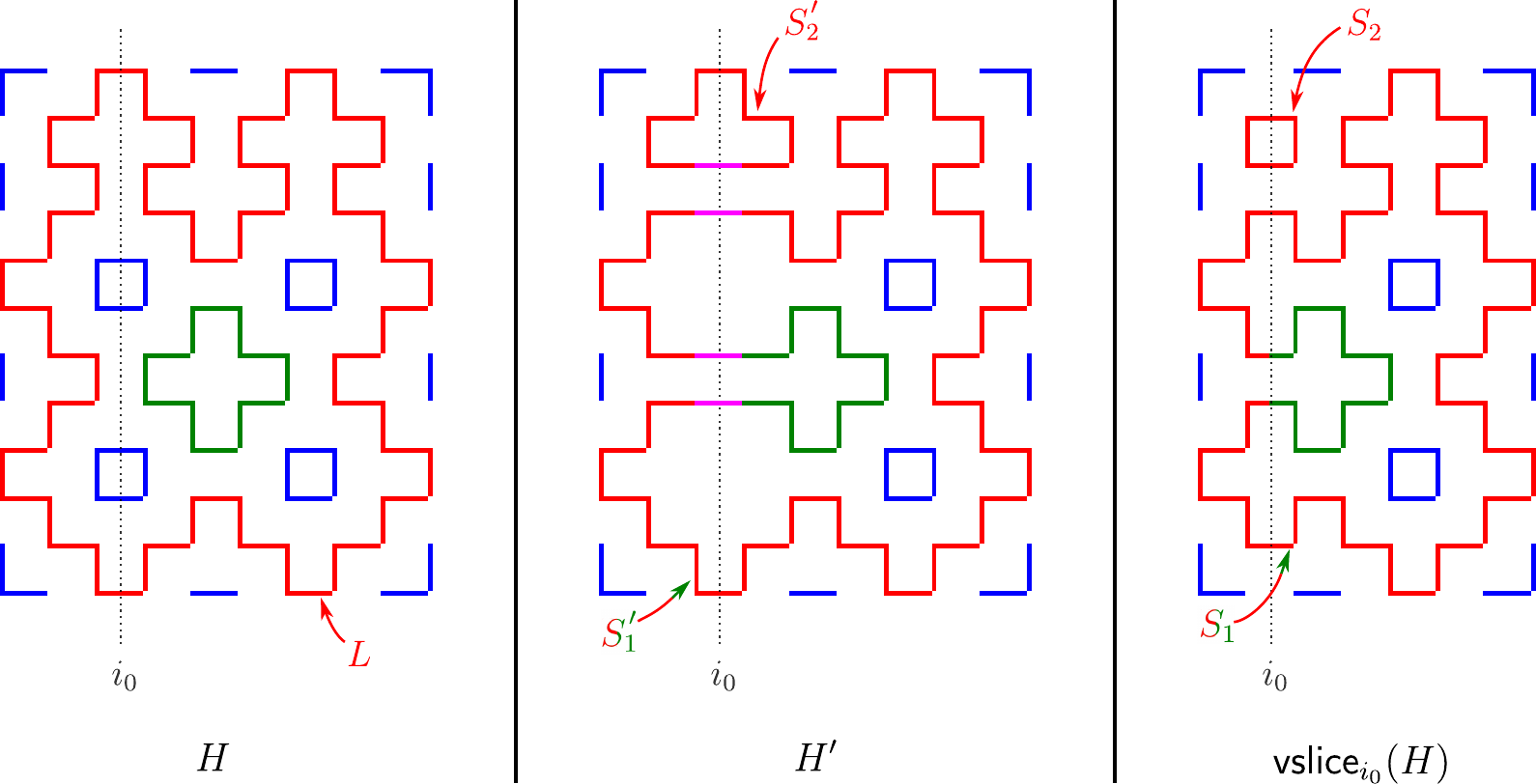}
  \caption{On the left is (part of) a hitomezashi pattern $H$. In the center is (part of) the pattern $H'$ obtained from $H$ by performing local moves and square deletions at longitude $i_0$. On the right is (part of) the hitomezashi pattern $\vslice_{i_0}(H)$. The post-slice components of the loop $L$ are $S_1$ and $S_2$. }\label{FigHito5}
\end{figure}

If a hitomezashi path $P$ lies either completely to the west or completely to the east of the deleted vertical strip, then its shape is unaffected by $\vslice_{i_0}$; the entire path $P$ just shifts to the east or the west.  The more interesting case is where $P$ does pass through the deleted vertical strip.  This is where it becomes important that we deleted two vertical grid lines with the \emph{same} label: the parts of $P$ on different sides of the slice still line up after the slicing; see Figure~\ref{FigHito5}.  It is possible for hitomezashi paths to merge with each other and/or break into more paths when we apply slicing, and our aim is to understand when and how this occurs.

It turns out that the slicing operation is easier to analyze if we first introduce \emph{local moves} that do not delete strips from the pattern. In this setting, we will actually need to consider curves that are not necessarily hitomezashi paths. Define a \dfn{grid path} to be an embedded curve in the plane that is contained in the union of all grid lines. We can orient grid paths just as we oriented hitomezashi paths; we make the convention that every closed grid path (i.e., grid loop) is oriented counterclockwise, so that its interior is on our left when we traverse it.

\begin{figure}[ht]
 \centering
 \includegraphics[height=1.52cm]{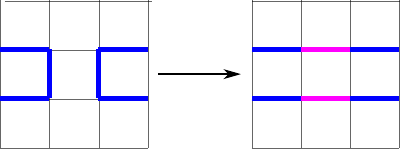}
  \caption{A local move.}\label{FigHito7}
\end{figure}

A \dfn{local move} consists of replacing two adjacent vertical stitches (with the same latitude and with longitudes $i_0-1/2$ and $i_0+1/2$) with two adjacent horizontal stitches as in Figure~\ref{FigHito7}. We apply such a local move at latitude $j$ if and only if there are vertical stitches at latitude $j$ with longitudes $i_0-1/2$ and $i_0+1/2$ and there are no horizontal stitches at longitude $i_0$ with latitude $j-1/2$ or $j+1/2$. If the two vertical stitches that are deleted come from different strands, then we call the local move a \dfn{splice}; if they come from the same strand, then we call the local move a \dfn{split}. We also define a \dfn{square deletion} to be the operation that deletes all stitches in a $4$-stitch square centered at longitude $i_0$. The image in the center of Figure~\ref{FigHito5} shows the pattern $H'$ that results from applying local moves and square deletions to the hitomezashi pattern $H$ on the left. The utility of these notions comes from the following proposition.

\begin{proposition}\label{prop:slicing-local-moves}
Let $H$ be a hitomezashi pattern, and let $H'$ be obtained from $H$ by applying all possible local moves and square deletions (at longitude $i_0$).  Then $\vslice_{i_0}(H)$ is equivalent to $H'$ up to homotopy.
\end{proposition}

\begin{proof}[Proof by picture]
At each latitude where $H$ has vertical stitches at longitudes $i_0-1/2$ and $i_0+1/2$, the patterns $H$, $H'$, and $\vslice_{i_0}(H)$ must locally be as depicted in one of the four cases in Figure~\ref{FigHito8}.  In each case, the homotopy equivalence is clear.  
\end{proof}

\begin{figure}[ht]
 \centering
 \includegraphics[height=6.26cm]{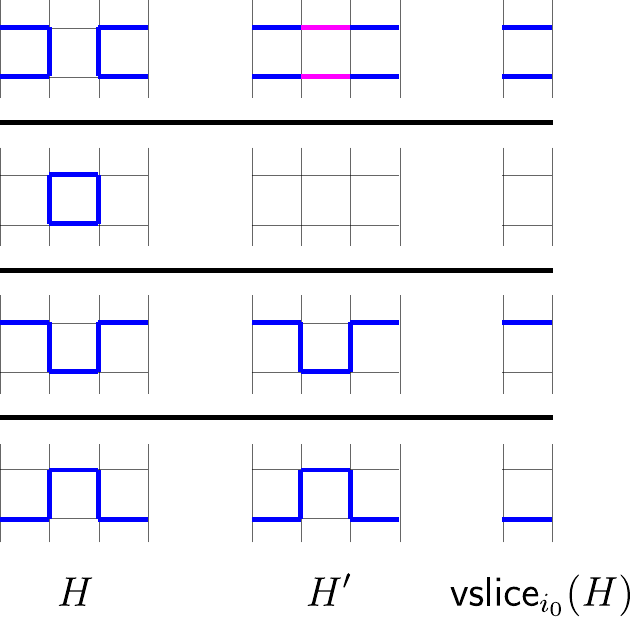}
  \caption{An illustration of Proposition~\ref{prop:slicing-local-moves}.}\label{FigHito8}
\end{figure}

Let $H$ be a hitomezashi pattern, and suppose we apply some number of local moves and square deletions to $H$ at longitude $i_0$. We call a grid path in the resulting pattern a \dfn{pseudo-hitomezashi path} (with respect to longitude $i_0$). Similarly, a \dfn{pseudo-hitomezashi loop} is a closed pseudo-hitomezashi path, and a \dfn{pseudo-hitomezashi strand} is a maximal pseudo-hitomezashi path (both still with respect to longitude $i_0$).   We can extend the parity argument of Lemma~\ref{Lem:Parity} to pseudo-hitomezashi paths.

\begin{lemma}\label{lem:parity-extended}
Let $P$ be an oriented pseudo-hitomezashi path with respect to longitude $i_0$. If $P$ has at least one vertical stitch at longitude $x$, then all of the stitches of $P$ at longitude $x$ are oriented in the same direction. Moreover, the orientations of the vertical stitches of $P$ at longitude $i_0-1/2$ are the opposite of the orientations of the vertical stitches of $P$ at longitude $i_0+1/2$.
\end{lemma}

\begin{proof}[Proof (sketch)]
They key observation is that $P$ is essentially the same as a hitomezashi path, except that there might be some latitudes at which $P$ has three consecutive horizontal stitches with longitudes $i_0-1,i_0,i_0+1$. The proof of the first statement is essentially the same as the proof of Lemma~\ref{Lem:Parity}; we leave the details to the reader. The second statement follows easily from our standing assumption that $\epsilon_{i_0-1/2}=\epsilon_{i_0+1/2}$.
\end{proof}

Let us say a pseudo-hitomezashi path $P$ is \dfn{agreeable} (for longitude $i_0$) if each of its vertical stitches with longitude $i_0-1/2$ (respectively, $i_0+1/2$) is oriented from north to south (respectively, south to north). Similarly, say $P$ is \dfn{anti-agreeable} (for longitude $i_0$) if each of its vertical stitches with longitude $i_0-1/2$ (respectively, $i_0+1/2$) is oriented from south to north (respectively, north to south). 
Lemma~\ref{lem:parity-extended} implies that every pseudo-hitomezashi path is either agreeable or anti-agreeable.  In particular, checking the orientation of a single vertical stitch at longitude $i_0-1/2$ or $i_0+1/2$ lets us determine which of these two possibilities occurs.

Recall that if $C$ is a closed curve, then $\overline{C}$ denotes the union of $C$ and its interior.

\begin{lemma}\label{prop:outdent-deletion-splice}
Suppose we perform a splice at longitude $i_0$ that makes an agreeable pseudo-hitomezashi loop $L$ merge with a pseudo-hitomezashi strand $S$. Let $L'$ be the pseudo-hitomezashi strand that results. Then both $S$ and $L'$ are contained in $\overline{L}$, and $L'$ is an agreeable pseudo-hitomezashi loop.
\end{lemma}

\begin{proof}
Let $s$ and $t$ be the stitches in $L$ and $S$, respectively, that are replaced in the application of the splice.  Without loss of generality, $s$ and $t$ are at longitudes $i_0-1/2$ and $i_0+1/2$, respectively  (otherwise, reflect everything across a vertical axis).  Note also that $s$ and $t$ have the same latitude.  Since $L$ is agreeable, $s$ is oriented from north to south.  This implies that the cell that has $s$ and $t$ as two of its opposite sides must lie in the interior of $L$. Consider the line segment $\Gamma$ whose endpoints are the midpoints of $s$ and $t$.  Then one endpoint of $\Gamma$ is in $L$, one endpoint of $\Gamma$ is in $S$, and the rest of $\Gamma$ lies in the interior of $L$. This implies that $S$ is contained in the interior of $L$. It follows that $S$ and $L'$ are both contained in $\overline{L}$.  In particular, since there are only finitely many stitches in $\overline{L}$, we see that $L'$ is a pseudo-hitomezashi loop. Note that $\Gamma$ passes through the exterior of $S$; this forces $t$ to be oriented from north to south. The orientations of the stitches in $S$ (other than $t$) get reversed when $L$ merges with $S$.  See Figure~\ref{FigHito9}.

It remains to show that $L'$ is agreeable. If $L'$ has no vertical stitches at longitude $i_0-1/2$ or $i_0+1/2$, then it is vacuously agreeable.  Otherwise, by the observations made in the paragraph immediately preceding this lemma, it suffices to check the orientation of a single vertical stitch of $L'$ at longitude $i_0-1/2$ or $i_0+1/2$; let $s_0$ be such a stitch. For the sake of convenience, we will consider the case where $\lon(s_0)=i_0+1/2$; the case where $\lon(s_0)=i_0-1/2$ is virtually identical. We need to check that $s_0$ is oriented from south to north in $L'$.  If $s_0$ is also a stitch of $L$, then it is oriented from south to north in $L'$ because $L$ is agreeable and because the orientations of stitches of $L$ do not change when $L$ merges with $S$.  (This is the case because the entire exterior of $L$ is also exterior to $L'$.)  Now suppose $s_0$ is not a stitch of $L$; then it must be a stitch of $S$.  We have observed that every pseudo-hitomezashi path is agreeable or anti-agreeable. Since $t$ is oriented from north to south in $S$, we know that $S$ is anti-agreeable, and in particular $s_0$ is oriented from north to south in $S$.  Since the orientations of the stitches in $S$ (other than $t$) get reversed when $L$ merges with $S$, we conclude that $s_0$ is oriented from south to north in $L'$, as desired.
\end{proof}

\begin{figure}[ht]
 \centering
 \includegraphics[height=3.24cm]{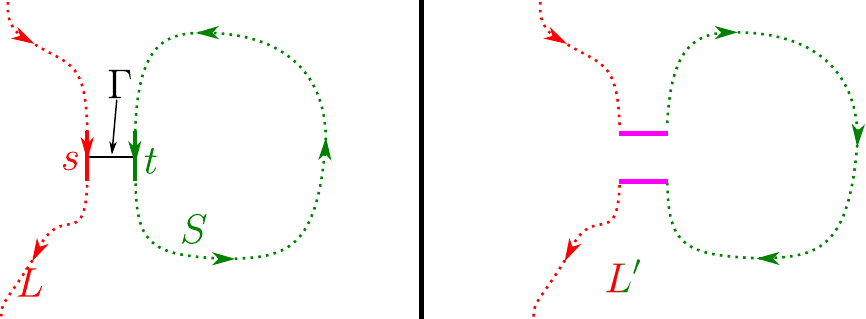}
  \caption{An illustration of Lemma~\ref{prop:outdent-deletion-splice}. The dotted curvy curves are schematic illustrations of the relevant grid paths.}\label{FigHito9}
\end{figure}

\begin{lemma}\label{lem:splitting}
Let $L$ be an agreeable pseudo-hitomezashi loop with at least one location where we can perform a split.  Let $L_1$ and $L_2$ denote the pseudo-hitomezashi strands into which $L$ breaks when we apply this split.  Then $L_1$ and $L_2$ are agreeable pseudo-hitomezashi loops, and neither is contained in the interior of the other.
\end{lemma}

\begin{proof}
Let $s$ and $s'$ be the two vertical stitches of $L$ at the location of the split, and let $t$ and $t'$ be the horizontal stitches that replace $s$ and $s'$ when we perform the split. Because $L$ is agreeable, the (closed) cell $\alpha$ enclosed by $s,s',t,t'$ is in $\overline L$. It follows that $\overline L_1\cup\overline L_2\cup\alpha=\overline L$, where $\overline L_1$, $\overline L_2$, and $\alpha$ are disjoint outside of the stitches $t$ and $t'$. Neither of $L_1$ and $L_2$ is contained in the interior of the other. Furthermore, all of the stitches in $L$ other than $s$ and $s'$ preserve their orientations in $L_1$ and $L_2$. Hence, $L_1$ and $L_2$ are agreeable pseudo-hitomezashi loops. See Figure~\ref{FigHito10}.
\end{proof}

\begin{figure}[ht]
 \centering
 \includegraphics[height=3.22cm]{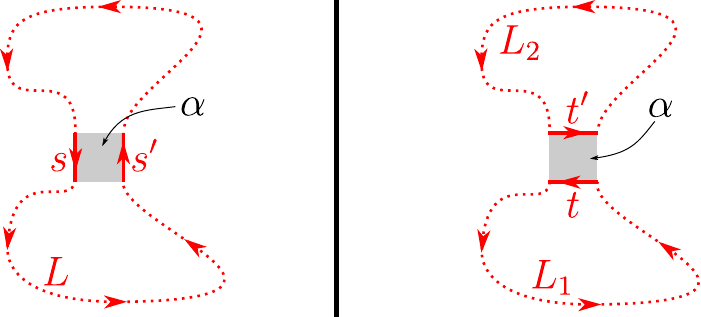}
  \caption{An illustration of Lemma~\ref{lem:splitting}. The dotted curvy curves are schematic illustrations of the relevant grid paths.}\label{FigHito10}
\end{figure}

We are finally ready to characterize how loops merge and break up under slicing.  For $L$ a hitomezashi loop in a hitomezashi pattern $H$, let $S_1,\ldots,S_r$ denote the distinct hitomezashi strands in $\vslice_{i_0}(H)$ that contain stitches from $\vslice_{i_0}(L)$.  (We know that $r$ is finite since $L$ has only finitely many stitches.) We call $S_1,\ldots,S_r$ the \dfn{post-slice components} of $L$. See Figure~\ref{FigHito5}.

\begin{proposition}\label{prop:outdent-deletion-split}
Let $L$ be a hitomezashi loop in a hitomezashi pattern $H$, and suppose that $i_0$ is an outdent longitude for $L$. Suppose $L$ splits into post-slice components $S_1,\ldots,S_r$ when we apply $\vslice_{i_0}$ to $H$. Then $S_1,\ldots,S_r$ are all hitomezashi loops contained in $\vslice_{i_0}(\overline L)$, and none of $S_1,\ldots,S_r$ is contained in the interior of another.
\end{proposition}

\begin{proof}
Using Lemma~\ref{lem:parity-extended} and the hypothesis that $i_0$ is an outdent longitude for $L$, we see that $L$ is agreeable. Let $H'$ denote the pattern of pseudo-hitomezashi strands that is obtained from $H$ by applying all local moves and square deletions at longitude $i_0$.  Let $S'_1, \ldots, S'_r$ denote the pseudo-hitomezashi strands in $H'$ corresponding to $S_1, \ldots, S_r$, as guaranteed by Proposition~\ref{prop:slicing-local-moves}.
Lemmas~\ref{prop:outdent-deletion-splice} and~\ref{lem:splitting} guarantee that $S_1',\ldots,S_r'$ are contained in $\overline L$, so they can be obtained from $L$ by applying finitely many local moves. Iterative applications of the same lemmas also guarantee that $S_1',\ldots,S_r'$ are pseudo-hitomezashi loops contained in $\overline L$ and that none of them is contained in the interior of another. The desired result follows from another application of Proposition~\ref{prop:slicing-local-moves}.
\end{proof}

\begin{remark}
There is an obvious analogue of Proposition~\ref{prop:outdent-deletion-split} for horizontal slicing. Both versions will be heavily used in the next section. 
\end{remark}

One could also consider slicing at an indent longitude or latitude. However, Proposition~\ref{prop:outdent-deletion-split} is not true as written when we replace outdent slicing with indent slicing; for this reason, it is important that we use outdent slicing (rather than indent slicing) in our later arguments. It appears that indent slicing exhibits a  ``dual'' behavior to outdent slicing; for indent slicing, for instance, a loop can merge only with a strand that is \emph{not} contained in its interior.  One new complication that arises in formulating a dual version of Proposition~\ref{prop:outdent-deletion-split} is that it might be possible for a loop to merge with an infinite strand.  We leave this line of inquiry as an open problem.

We end this section with one final lemma about slicing operations that will be useful for proving Theorems~\ref{thm:possible-length} and~\ref{thm:possible-area}. 

\begin{lemma}\label{lem:5am}
Let $L$ be a hitomezashi loop in a hitomezashi pattern $H$. Let $i_0$ be the largest (half-integer) longitude at which there is a horizontal stitch that is either an outdent or an indent of $L$. Then $i_0$ is an outdent longitude of $L$.
\end{lemma}

\begin{proof}
As usual, we orient $L$ counterclockwise. Suppose instead that $i_0$ is an indent longitude of $L$, and let $s$ be the south-most horizontal stitch of $L$ with longitude $i_0$. Consider traversing $L$ according to its orientation. Let $t$ be the stitch immediately after $s$ in this traversal. Let $s'$ be the first (horizontal) stitch with longitude $i_0$ that we reach after $s$ during this traversal. It follows from the choice of $s$ that the unit grid cell immediately south of $s$ is in the exterior of $L$, so $s$ is oriented from west to east. This implies that all stitches in $L_{[t,s']}$ except for $s'$ have longitudes strictly greater than $i_0$, so it follows from the definition of $i_0$ that $L_{[t,s']}$ cannot cross through an outdent longitude of $L$. Since $i_0$ is an indent longitude of $L$, we have $\lat(t)=\lat(s)-1/2$, and $t$ must be oriented from north to south. The choice of $s$ guarantees that $\lat(s)<\lat(s')$, so the path $L_{[t,s']}$, which starts off traveling south, must have some stitch oriented from south to north. But this forces $L_{[t,s']}$ to cross through an outdent longitude of $L$, which is a contradiction.  
\end{proof}

\begin{remark}
In an earlier version of this paper, we used the machinery of slicing to give an alternative proof of Pete's Theorems~\ref{ThmExtremalMatch}, \ref{ThmTwoStitches}, and~\ref{ThmOdd}.  Our argument followed a very complicated double-induction scheme; although certainly less elegant than Pete's Theorem~\ref{thm:pete}, it demonstrated the versatility of the slicing method.
\end{remark}

\section{Lengths and areas}\label{sec:length-area}
Our goal in this section is to prove Theorems~\ref{thm:possible-length} and~\ref{thm:possible-area}, which characterize the possible lengths and areas of hitomezashi loops.  The properties of slicing from the last section allow us to use inductive arguments.
%
When we study the effect that slicing has on a single strand, we also need to consider the other strands with which it merges, and the strands with which those strands merge, and so on.  This idea leads to the following definition.  Let $H$ be a hitomezashi pattern, and let $i_0$ be a (half-integer) longitude such that $\epsilon_{i_0-1/2}=\epsilon_{i_0+1/2}$.  For strands $S_1$ and $S_2$ in $H$, write $S_1 \sim' S_2$ if there exists a strand in $\vslice_{i_0}(H)$ that contains stitches from both $S_1$ and $S_2$.  Define the \dfn{intertwining equivalence relation} (at longitude $i_0$) to be the transitive closure $\sim$ of the relation $\sim'$, and say $S_1$ and $S_2$ are \dfn{intertwined} if $S_1 \sim S_2$.
We can of course also make the analogous definition for horizontal slicing.


\begin{proof}[Proof of Theorem~\ref{thm:possible-length}]
We proceed by induction on width.  For the base case, note that Theorem~\ref{thm:possible-length} holds trivially for the unique hitomezashi loop of width $1$ (viz., the square).  Now let $L$ be a hitomezashi loop of width $w_0$ in a hitomezashi pattern $H$, and assume inductively that Theorem~\ref{thm:possible-length} holds for all loops of width strictly smaller than $w_0$.  Let $i_0$ be an outdent longitude for $L$.  Let $\{L_1, \ldots, L_q\}$ be the intertwining equivalence class of $L$ (at longitude $i_0$), with the convention $L_1=L$; Lemmas~\ref{prop:outdent-deletion-splice} and~\ref{lem:splitting} guarantee that $q$ is finite and that $L_2, \ldots, L_q$ are all hitomezashi loops contained in $\overline{L}$.  Let $M_1, \ldots, M_r$ be the post-slice components of $L_1, \ldots, L_q$; Proposition~\ref{prop:outdent-deletion-split} guarantees that $M_1, \ldots, M_r$ are all hitomezashi loops contained in $\vslice_{i_0}(\overline{L})$.

Let $a$ denote the total number of local moves that involve stitches from $L_1, \ldots, L_q$.  The number of splits minus the number of splices is $q-r$ (if the local moves are applied in any fixed order); in particular, $a$ and $q-r$ have the same parity.  Let $b$ denote the number of horizontal stitches of $L_1, \ldots, L_q$ at longitude $i_0$; it is clear that $b$ is even.  We have
\begin{equation}\label{eq:length}
\length(L_1)+\cdots+\length(L_q)=\length(M_1)+\cdots+\length(M_r)+4a+4b  
\end{equation}
since all stitches of $L_1, \ldots, L_q$ are preserved in $M_1, \ldots, M_r$ except for the stitches between the longitudes $i_0-1$ and $i_0+1$.
We know by induction that $\length(L_2)\equiv\cdots\equiv\length(L_q)\equiv\length(M_1)\equiv\cdots\equiv\length(M_r)\equiv 4\pmod 8$. Since $a \equiv q-r \pmod{2}$ and $b$ is even, we deduce from \eqref{eq:length} that
\[\length(L_1)+4(q-1)\equiv 4r+4(q-r) \pmod{8}.\]
We conclude that
$\length(L_1)\equiv 4 \pmod{8}$,
as desired.
\end{proof}

One could use a complicated version of the above argument to prove Theorem~\ref{thm:possible-area}, but it is simpler to deduce Theorem~\ref{thm:possible-area} from Theorem~\ref{thm:possible-length}.

\begin{proof}[Proof of Theorem~\ref{thm:possible-area}]
We will apply Pick's Theorem to the polygonal region bounded by $L$.  The lattice points in the interior of $L$ can be partitioned according to the hitomezashi loops on which they lie; since the number of lattice points on each such loop is a multiple of $4$ (from Theorem~\ref{thm:possible-length}, say, although this is also not difficult to see directly), we know that the total number $X$ of lattice points in the interior of $L$ is a multiple of $4$.  Theorem~\ref{thm:possible-length} also tells us that the number $Y$ of lattice points on the loop $L$ is congruent to $4$ modulo $8$.  By Pick's Theorem, the area of the region bounded by $L$ is $X+Y/2-1$, and this is congruent to $1$ modulo $4$, as desired.
\end{proof}

\section{Counting loops and regions}\label{sec:counting}
Now that we have studied structural properties of hitomezashi loops, we  turn to some more enumerative/probabilistic problems. The central problem is to understand the number of loops and/or regions that appear in an $m\times n$ hitomezashi pattern $H$. Note that the number of regions of $H$ is equal to the sum of the number of loops of $H$ and the number of regions that touch the boundary of the $m\times n$ grid. The number of such boundary regions is at most $2(m+n)$, and this quantity will turn out to be negligible in the settings that concern us; hence, we will just focus on counting loops. As mentioned in the introduction, an interesting problem is to determine the behavior of $\mathbb E[\Loops({\mathcal H}_{1/2}(m,n))]$, the expected number of loops in an $m \times n$ hitomezashi pattern when the labels $\epsilon_i$ and $\eta_j$ are chosen independently at random and each label has an equal probability of being $0$ and $1$.  We will prove that as $m$ and $n$ tend to infinity, this expected value grows like $\frac{\pi^2-9}{12}mn$. We first need the following lemma. Recall that ${\bf L}$ is the set of all hitomezashi loops modulo translation. Let \[\mathbf{c}=\sum_{L\in{\bf L}}2^{-(\width(L)+\height(L)+2)}.\]  


\begin{lemma}\label{PropEmn}
As $\min \{m,n\}\to\infty$, we have \[\mathbb E[\Loops({\mathcal H}_{1/2}(m,n))]=(\mathbf{c}+o(1))mn.\] 
\end{lemma}
\begin{proof}
We can write down an exact formula for $\mathbb E[\Loops({\mathcal H}_{1/2}(m,n))]$.  If $L \in \mathbf{L}$ has width greater than $n$ or height greater than $m$, then it cannot appear as a loop in an $m \times n$ hitomezashi pattern.  If $L \in \mathbf{L}$ has width $w \leq n$ and height $h \leq m$, then there are exactly $(m+1-h)(n+1-w)$ translates of $L$ that could appear in an $m \times n$ hitomezashi pattern.  When $H$ is a randomly-chosen $m \times n$ hitomezashi pattern, each of these translates of $L$ appears with probability $2^{-(w+h+2)}$.  By linearity of expectation, the expected number of translates of $L$ is $$2^{-(w+h+2)}(m+1-h)(n+1-w).$$
Let $\mathbf{L}_{\leq m, \leq n}$ denote the set of loops in $\mathbf{L}$ with height at most $m$ and width at most $n$.  Summing over all loops in $\mathbf{L}$, we find that
\begin{equation}\label{eqn:exact-count}
\mathbb E[\Loops({\mathcal H}_{1/2}(m,n))]=\sum_{L \in \mathbf{L}_{\leq m, \leq n}}2^{-(\width(L)+\height(L)+2)}(m+1-\height(L))(n+1-\width(L)).
\end{equation}
This is the promised exact formula. It follows immediately from this formula that $E_{m,n}\leq \mathbf{c}mn$. Using the positivity of all summands involved, we note that for $\min\{m,n\} \geq r^2$, we have \begin{align*}
    \frac{\mathbb E[\Loops({\mathcal H}_{1/2}(m,n))]}{mn} &\geq\sum_{L \in \mathbf{L}_{\leq r, \leq r}}2^{-(\width(L)+\height(L)+2)}\left(1-\frac{\height(L)-1}{m}\right)\left(1-\frac{\width(L)-1}{n}\right) \\ 
    &\geq\left(\sum_{L \in \mathbf{L}_{\leq r, \leq r}}2^{-(\width(L)+\height(L)+2)}\right)(1-1/r)^2.
\end{align*} 
Hence,
\begin{equation}\label{eqn:c-ineq}
\left(\sum_{L \in \mathbf{L}_{\leq r, \leq r}}2^{-(\width(L)+\height(L)+2)}\right)(1-1/r)^2mn\leq \mathbb E[\Loops({\mathcal H}_{1/2}(m,n))] \leq \mathbf{c}mn.
\end{equation}

Since $\mathbb E[\Loops({\mathcal H}_{1/2}(m,n))]$ is trivially bounded above by $mn$, the first inequality in \eqref{eqn:c-ineq} gives
\[\sum_{L \in \mathbf{L}_{\leq r, \leq r}}2^{-(\width(L)+\height(L)+2)}\leq (1-1/r)^{-2},\]
so the infinite sum defining $\mathbf{c}$ converges.  As $r\to\infty$, the coefficient of $mn$ on the left-hand side of \eqref{eqn:c-ineq} approaches $\mathbf{c}$, so $\mathbb E[\Loops({\mathcal H}_{1/2}(m,n))]=(\mathbf{c}+o(1))mn$. 
\end{proof}

Together with Lemma~\ref{PropEmn}, the following proposition completes the proof of Theorem~\ref{thm:expected-regions}.
\begin{theorem}\label{ThmEmn}
We have
$$\mathbf{c}=\frac{\pi^2-9}{12}.$$
\end{theorem}

\begin{proof}
Let $D(n,k)$ denote the number of Dyck paths of semilength $n$ with height $k$ (with the convention that $D(0,0)=1$). It follows immediately from Theorem~\ref{thm:pete} that \[{\bf c}=\sum_{k\geq 0}\sum_{a,b\geq 0}D(a,k)D(b,k)\left(\frac{1}{2}\right)^{2a+2b+4}.\] Let $F_k(x)=\sum_{n\geq 0}D(n,k)x^k$. Then \[F_k(x)^2=\sum_{a,b\geq 0}D(a,k)D(b,k)x^{a+b},\] so ${\bf c}=\frac{1}{16}\sum_{k\geq 0}F_k(1/4)^2$. Kreweras \cite[Page~37]{Kreweras} showed that \begin{equation}\label{Eq:Kreweras}
    F_k(x)=\frac{x^k}{f_k(x)f_{k+1}(x)},
\end{equation} where the series $f_m(x)$ are defined by the recurrence $f_m(x)=f_{m-1}(x)-xf_{m-2}(x)$ subject to the initial conditions $f_0(x)=f_1(x)=1$. We claim that $f_m(1/4)=(m+1)/2^m$ for all $m \geq 0$.  The claim clearly holds for $m=0,1$, and for larger $m$ we inductively compute
$$f_m(1/4)=\frac{m}{2^{m-1}}-\frac{1}{4} \cdot \frac{m-1}{2^{m-2}}=\frac{2m-(m-1)}{2^{m}}=\frac{m+1}{2^m}.$$
Combining this with \eqref{Eq:Kreweras} yields 
\begin{align*}
{\bf c}&=\frac{1}{16}\sum_{k\geq 0}F_k(1/4)^2 \\
&=\frac{1}{16}\sum_{k\geq 0}\left(\frac{(1/4)^k}{f_k(1/4)f_{k+1}(1/4)}\right)^2 \\
&=\frac{1}{16}\sum_{k\geq 0}\left(\frac{2}{(k+1)(k+2)}\right)^2 \\
&=\frac{1}{4}\sum_{k\geq 1}\frac{1}{k^2(k+1)^2} \\
&=\frac{1}{4}\sum_{k\geq 1}\left(\frac{1}{k^2}+\frac{1}{(k+1)^2}-2\left(\frac{1}{k}-\frac{1}{k+1}\right)\right) \\
&=\frac{1}{4}\sum_{k\geq 1}\frac{1}{k^2}+\frac{1}{4}\sum_{k\geq 1}\frac{1}{(k+1)^2}-\frac{1}{2}\sum_{k\geq 1}\left(\frac{1}{k}-\frac{1}{k+1}\right) \\
&=\frac{1}{4}\cdot\frac{\pi^2}{6}+\frac{1}{4}\left(\frac{\pi^2}{6}-1\right)-\frac{1}{2} \\
&=\frac{\pi^2-9}{12}. \qedhere
\end{align*}
\end{proof}

\section{Open Problems}\label{sec:open}
In this final section, we reiterate several of the open problems that have arisen in this paper, and we present other ideas for future work on hitomezashi patterns.

\subsection{Dual hitomezashi patterns}
Given a hitomezashi pattern $H$ with labels $\epsilon_i$ and $\eta_j$, we can define the \dfn{dual} hitomezashi pattern of $H$ to be the hitomezashi pattern $H^*$ with labels ${\epsilon}^*_i=1-\epsilon_i$ and $\eta^*_j=1-\eta_j$.  In other words, $H^*$ is obtained from $H$ by interchanging all of the stitches and non-stitches on grid lines.  See Figure~\ref{FigHito15}.  This notion first appeared in the paper of Hayes and Seaton \cite{HayesSeaton}, who described the dual as the result of flipping over the cloth on which the hitomezashi pattern is embroidered.  There are many questions that one could ask about the relationship between a hitomezashi pattern $H$ and its dual $H^*$.  For instance, in the case where $H$ is an $m \times n$ hitomezashi pattern, what is the relationship between $\Loops(H)$ and $\Loops(H^*)$?  If $L$ is a hitomezashi loop in $H$, then what can we say about the size, number, and structure of the hitomezashi loops in $H^*$ that are contained in $\overline{L}$?

\begin{figure}[ht]
 \centering
 \includegraphics[height=4.51cm]{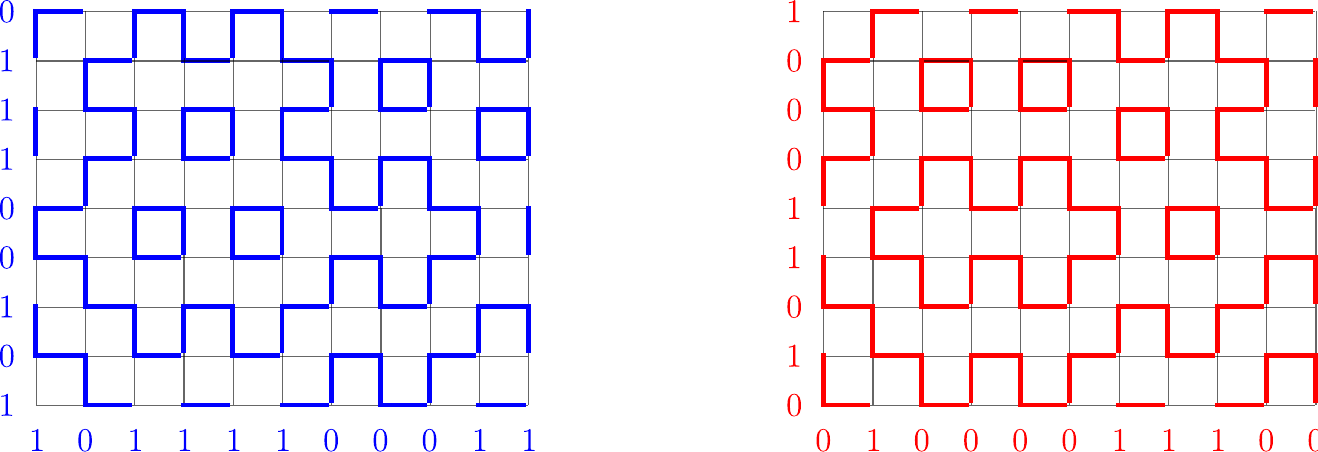}
  \caption{Dual hitomezashi patterns.}\label{FigHito15}
\end{figure}

\subsection{Enumeration of hitomezashi loops}

As described in Section~\ref{sec:pete}, Pete's results allow for the enumeration of hitomezashi loops (up to translation) by height and width.  It is also natural to consider the enumeration of hitomezashi loops according to other statistics such as length and, area.  Indeed, the asymptotic enumeration of polyominoes according to area and/or perimeter has attracted a great deal of attention \cite{Barequet, Conway, ConwayGuttmann, Golomb, GuttmannBook, Jensen, Klarner}.

\subsection{Indent slicing}
In Section~\ref{sec:slice}, we focused on the effect of slicing a hitomezashi loop at an outdent longitude, and we mentioned that the case of slicing at an indent longitude is potentially more complicated.  In particular, a (slightly technical) open problem is to formulate a version of Proposition~\ref{prop:outdent-deletion-split} for slicing a hitomezashi loop at an indent longitude.  What happens if we look at general strands in addition to loops?

\subsection{Other variations}
There are many possible variations on hitomezashi patterns; we list just a few here.
\begin{itemize}
    \item One place to start is generalizing hitomezashi patterns to higher dimensions.  It is not immediately clear what the most interesting generalization is.  In three dimensions, for instance, one could choose a $0/1$ label for each integer-valued grid line and alternate included and excluded line segments.  Each connected component of the resulting pattern would be a $1$-skeleton, and there would be $3$ line segments incident to each lattice point.  For another possibility, consider \[(\mathbb Z\times\mathbb R\times \mathbb R)\cup(\mathbb R\times\mathbb Z\times\mathbb R)\cup(\mathbb R\times\mathbb R\times\mathbb Z),\] which the union of grid planes in $\mathbb R^3$, each of which is parallel to one of the standard coordinate planes. One could choose a $0/1$ label for each such grid plane and draw one of two complementary unit checkerboards.  In this case, the connected components of the pattern would form surfaces, and the bounded surfaces arising in this manner would be the boundaries of special \emph{polycubes} (see, e.g., \cite{Gardner, GuttmannBook}).  It would be interesting to connect these patterns to triples of Dyck paths, in the style of~\cite{Pete}.
    \item One can also define hitomezashi patterns on the torus.  In this case, every stitch will be part of a loop.  What can we say about the structure of loops in toric hitomezashi patterns?
    \item Another variation, suggested by MacDonald, consists of placing stitches on a triangular grid; watch \cite{Numberphile} for more details about this idea.
\end{itemize}

\section*{Acknowledgements}
The first author is supported by an NSF Graduate Research Fellowship (grant DGE--1656466) and a Fannie and John Hertz Foundation Fellowship.  The second author is supported by an NSF Graduate Research Fellowship (grant DGE--2039656). We are grateful to Noga Alon and Susan Defant for helpful conversations.  We thank Omer Angel for drawing our attention to Pete's paper~\cite{Pete}.

\end{document}